\documentclass[a4paper,11pt]{amsart}
\usepackage{amssymb,mathtools,amsthm}
\usepackage[pdftex]{graphicx}
\usepackage[all]{xy}
\usepackage{ytableau}
\ytableausetup{centertableaux}
\usepackage{cleveref}
\usepackage{url}

\numberwithin{equation}{section}

\theoremstyle{plain}
\newtheorem{thm}{Theorem}[section]
\newtheorem{prop}[thm]{Proposition}
\newtheorem{lem}[thm]{Lemma}

\theoremstyle{definition}
\newtheorem{defn}[thm]{Definition}

\newtheorem{ex}[thm]{Example}
\newtheorem{rem}[thm]{Remark}

\newcommand{\Z}{\mathbb{Z}}

\newcommand{\bk}{{\boldsymbol{k}}}

\newcommand{\bm}{{\boldsymbol{m}}}
\newcommand{\bn}{{\boldsymbol{n}}}
\newcommand{\ba}{{\boldsymbol{a}}}
\newcommand{\bb}{{\boldsymbol{b}}}
\newcommand{\bone}{{\boldsymbol{1}}}
\newcommand{\bN}{{\boldsymbol{N}}}

\newcommand{\tJ}{\tilde{J}}
\newcommand{\tm}{\tilde{m}}
\newcommand{\tn}{\tilde{n}}

\newcommand{\oj}{\overline{\jmath}}

\newcommand{\coloniff}{\;:\Longleftrightarrow\;}

\DeclareMathOperator{\SSYT}{SSYT}

\title[Remarks on Maesaka--Seki--Watanabe's Formula]
{Some Remarks on Maesaka--Seki--Watanabe's Formula for the Multiple Harmonic Sums}
\author{Shuji Yamamoto}
\date{}

\thanks{This research was supported by JSPS KAKENHI JP18H05233 and JP21K03185. }

\address{
    Interdisciplinary Faculty of Science and Engineering, Shimane University, 
    1060 Nishi-Kawatsu, Matsue, 690-8504, Japan
}
\email{yamashu@riko.shimane-u.ac.jp}

\begin{document}

\begin{abstract}
Recently, Maesaka, Seki and Watanabe discovered a surprising equality 
between multiple harmonic sums and certain Riemann sums which 
approximate the iterated integral expression of the multiple zeta values. 
In this paper, we describe the formula corresponding to the multiple zeta-star values and, 
more generally, to the Schur multiple zeta values of diagonally constant indices. 
We also discuss the relationship of these formulas with Hoffman's duality identity 
and an identity due to Kawashima. 
\end{abstract}

\maketitle

\section{Introduction}
For a finite tuple $\bk=(k_1,\ldots,k_r)$ of positive integers, 
called an \emph{index}, we define the \emph{multiple harmonic sums} 
\[\zeta_{<N}(\bk)\coloneqq
\sum_{0<m_1<\cdots<m_r<N}\frac{1}{m_1^{k_1}\cdots m_r^{k_r}},\]
where $N$ is any positive integer. 
By convention, we set $\zeta_{<N}(\varnothing)\coloneqq 1$ for the empty index $\varnothing$ 
(i.e., $r=0$). 
When the index $\bk$ is \emph{admissible}, i.e., $k_r\ge 2$ or $\bk=\varnothing$, the limit 
\[\zeta(\bk)\coloneqq \lim_{N\to\infty}\zeta_{<N}(\bk)
=\sum_{0<m_1<\cdots<m_r}\frac{1}{m_1^{k_1}\cdots m_r^{k_r}}\]
is called the \emph{multiple zeta value}. 

Recently, Maesaka--Seki--Watanabe \cite{MSW} introduced another kind of finite sum: 
\begin{equation}\label{eq:flat sum}
\zeta^\flat_{<N}(\bk)\coloneqq 
\sum_{\substack{0<n_{i1}\le\cdots\le n_{ik_i}<N\;(1\le i\le r)\\
n_{ik_i}<n_{(i+1)1}\;(1\le i<r)}}
\prod_{i=1}^r\frac{1}{(N-n_{i1})n_{i2}\cdots n_{ik_i}}. 
\end{equation}
For example, we have 
\[\zeta^\flat_{<N}(2,1,3)=
\sum_{0<n_1\le n_2<n_3<n_4\le n_5\le n_6<N}
\frac{1}{(N-n_1)n_2\cdot (N-n_3)\cdot (N-n_4)n_5n_6}. \]
Note that this sum can be written as 
\[\zeta^\flat_{<N}(2,1,3)=\frac{1}{N^6}
\sum_{0<n_1\le n_2<n_3<n_4\le n_5\le n_6<N}
\frac{1}{(1-\frac{n_1}{N})\frac{n_2}{N}\cdot (1-\frac{n_3}{N})\cdot
(1-\frac{n_4}{N})\frac{n_5}{N}\frac{n_6}{N}}, \]
which is a Riemann sum which approximates the well-known 
iterated integral expression 
\[\zeta(2,1,3)=\int_{0<x_1<x_2<x_3<x_4<x_5<x_6<1}
\frac{dx_1}{1-x_1}\frac{dx_2}{x_2}\cdot\frac{dx_3}{1-x_3}\cdot
\frac{dx_4}{1-x_4}\frac{dx_5}{x_5}\frac{dx_6}{x_6}. \]
Thus we have 
\[\lim_{N\to\infty}\zeta_{<N}(2,1,3)=\zeta(2,1,3)
=\lim_{N\to\infty}\zeta^\flat_{<N}(2,1,3), \]
and the same holds for any admissible index. 
The surprising discovery of Maesaka--Seki--Watanabe \cite{MSW} is that 
the equality holds before taking the limit: 

\begin{thm}[{\cite[Theorem 1.3]{MSW}}]\label{thm:MSW}
For any index $\bk$ and any integer $N>0$, we have 
\begin{equation}\label{eq:MSW}
\zeta_{<N}(\bk)=\zeta_{<N}^\flat(\bk). 
\end{equation}
\end{thm}

In the following, we call this formula \eqref{eq:MSW} the \emph{MSW formula}. 
One may expect various applications and generalizations. 
In addition to the proof of duality relations given in \cite{MSW}, 
an application has been given by Seki \cite{S}, 
who provided a new proof of the extended double shuffle relation of multiple zeta values. 
Hirose, Matsusaka and Seki \cite{HMS} generalize the MSW formula to the case of multiple polylogarithms. 

The purpose of the present article is to show some results related to the MSW formula. 
First, in \S2, we describe the ``star version'' of it, that is, 
a formula of the same type for the \emph{multiple star harmonic sum} 
\[\zeta^\star_{<N}(\bk)\coloneqq 
\sum_{0<m_1\le\cdots\le m_r<N}\frac{1}{m_1^{k_1}\cdots m_r^{k_r}}.\]
The result (\Cref{thm:MSW star}) is a discrete analogue of the 2-poset integral \cite[Corollary 1.3]{Y} 
for the multiple zeta-star value. 
We will explain how this star version is deduced from the non-star version \eqref{eq:MSW} and vice versa. 

In \S3, we prove a generalization of both non-star and star MSW formulas, 
namely, the formula for \emph{Schur multiple harmonic sums with diagonally constant indices}. 
The result is a discretization of the integral expression given by Hirose--Murahara--Onozuka \cite{HMO}. 
Our proof is a natural generalization of that of \eqref{eq:MSW} 
given in \cite{MSW}, the so-called ``connector method''. 

In \S4, we examine Hoffman's duality identity 
\[\zeta_{<N}^\star(\bk^\vee)=H_{<N}(\bk)\]
in the light of the MSW formula 
(see \S4 for the definition of $\bk^\vee$ and $H_{<N}$). 
The combination of this identity and the star MSW formula provides 
an expression for $H_{<N}(\bk)$. 
We notice that this expression can be proven directly, applying the computation in \S3, 
and hence a new proof of Hoffman's identity is obtained. 

Finally, in \S5, we explain the relationship of the star MSW formula 
with an identity due to Kawashima \cite{K2}. 
We see that Kawashima's identity is, in a sense, equivalent to the star MSW formula. 

\subsection*{Notation}
For $m,n\in\Z$, we set $[m,n]\coloneqq\{a\in\Z\mid m\le a\le n\}$. 
We call such a subset of $\Z$ an \emph{interval} in $\Z$, or simply an interval. 
In particular, the empty set is an interval since it is written as $\emptyset=[m,n]$ with $m>n$. 
The empty set in a general context is denoted by $\emptyset$, while the empty \emph{index} is denoted by $\varnothing$. 

\section{MSW formula for multiple star harmonic sums}

Recall that, for an index $\bk=(k_1,\ldots,k_r)$ and a positive integer $N$, 
the multiple star harmonic sum is defined by 
\[\zeta^\star_{<N}(\bk)\coloneqq 
\sum_{0<m_1\le\cdots\le m_r<N}\frac{1}{m_1^{k_1}\cdots m_r^{k_r}}. \]
We also define the $\flat$-sum 
\begin{equation}\label{eq:star flat sum}
\zeta_{<N}^{\star\flat}(\bk)\coloneqq 
\sum_{\substack{0<n_{i1}\le\cdots\le n_{ik_i}<N\,(1\le i\le r)\\
n_{(i-1)1}\le n_{ik_i}\,(2\le i\le r)}}
\prod_{i=1}^r\frac{1}{(N-n_{i1})n_{i2}\cdots n_{ik_i}}. 
\end{equation}
For example, 
\[\zeta_{<N}^{\star\flat}(2,1,3)
=\sum_{n_1\le n_2\le n_3\ge n_4\ge n_5\le n_6}
\frac{1}{(N-n_1)n_2n_3\cdot (N-n_4)\cdot (N-n_5)n_6}\]
(here, we omit from the notation the condition $0<n_i<N$ for the running variables $n_i$). 
As before, we set $\zeta^\star_{<N}(\varnothing)=\zeta_{<N}^{\star\flat}(\varnothing)=1$. 

It is well known that the multiple harmonic and star harmonic sums satisfy 
the ``antipode identity'': For $r>0$, 
\begin{equation}\label{eq:antipode}
\sum_{i=0}^r(-1)^i\zeta_{<N}(k_1,\ldots,k_i)\cdot\zeta_{<N}^\star(k_r,\ldots,k_{i+1})=0. 
\end{equation}
Notice that the $\flat$-version 
\begin{equation}\label{eq:antipode flat}
\sum_{i=0}^r(-1)^i\zeta_{<N}^\flat(k_1,\ldots,k_i)\cdot\zeta_{<N}^{\star\flat}(k_r,\ldots,k_{i+1})=0
\end{equation}
also holds for $r>0$. In fact, for each $i$, we have 
\[\zeta_{<N}^\flat(k_1,\ldots,k_i)\cdot\zeta_{<N}^{\star\flat}(k_r,\ldots,k_{i+1})
=\sum_{\substack{0<n_{j1}\le\cdots\le n_{jk_j}<N\,(1\le j\le r)\\
n_{jk_j}<n_{(j+1)1}\,(1\le j<i)\\
n_{jk_j}\ge n_{(j+1)1}\,(i<j<r)}}
\prod_{j=1}^r\frac{1}{(N-n_{j1})n_{j2}\cdots n_{jk_j}}. \]
If we decompose the latter sum into two parts according to whether $n_{ik_i}<n_{(i+1)1}$ 
or $n_{ik_i}\ge n_{(i+1)1}$, 
the left-hand side of \eqref{eq:antipode flat} becomes a telescopic sum and the equality follows. 

The analogue of \Cref{thm:MSW} for the multiple star harmonic sum is the following: 

\begin{thm}\label{thm:MSW star}
For any index $\bk$ and any integer $N>0$, we have 
\begin{equation}\label{eq:MSW star}
\zeta_{<N}^\star(\bk)=\zeta_{<N}^{\star\flat}(\bk). 
\end{equation}
\end{thm}
\begin{proof}
This follows from \Cref{thm:MSW} by induction on the depth (i.e., the length) $r$ of 
the index $\bk$. 
The formula trivially holds for $\bk=\varnothing$. 
When $r>0$, we have two identities 
\begin{align*}
\zeta_{<N}^\star(k_r,\ldots,k_1)&=
-\sum_{i=1}^r(-1)^i\zeta_{<N}(k_1,\ldots,k_i)\cdot\zeta_{<N}^\star(k_r,\ldots,k_{i+1}),\\
\zeta_{<N}^{\star\flat}(k_r,\ldots,k_1)&=
-\sum_{i=1}^r(-1)^i\zeta_{<N}^{\flat}(k_1,\ldots,k_i)\cdot\zeta_{<N}^{\star\flat}(k_r,\ldots,k_{i+1})
\end{align*}
by \eqref{eq:antipode} and \eqref{eq:antipode flat}, and 
they are equal by \eqref{eq:MSW} and the induction hypothesis. 
\end{proof}

Conversely, \Cref{thm:MSW} can be deduced from \Cref{thm:MSW star} in the same way. 
In other words, Theorems \ref{thm:MSW} and \ref{thm:MSW star} are equivalent 
once the identities \eqref{eq:antipode} and \eqref{eq:antipode flat} are provided.

\section{Schur multiple harmonic sums}

Following Nakasuji--Phuksuwan--Yamasaki \cite{NPY}, we define the (skew-)Schur multiple harmonic sum by 
\begin{equation}\label{eq:Schur MHS}
\zeta_{<N}(\bk)\coloneqq 
\sum_{(m_{ij})\in\SSYT_{<N}(D)}\prod_{(i,j)\in D}\frac{1}{m_{ij}^{k_{ij}}}. 
\end{equation}
Here $\bk$ is an index on a skew Young diagram $D$ 
(that is, $\bk=(k_{ij})$ is a tuple of positive integers indexed by $(i,j)\in D$), 
and $\SSYT_{<N}(D)$ denotes the set of semi-standard Young tableaux on $D$ 
whose entries are positive integers less than $N$. 
For example, 
\[\zeta_{<N}\left(\ytableausetup{boxsize=1.5em}
\begin{ytableau}
k_{11} & k_{12} & k_{13}\\
k_{21} & k_{22} 
\end{ytableau}\right)
=\sum_{\setlength{\arraycolsep}{1pt}\scriptsize
\begin{array}{ccccc}m_{11}&\le&m_{12}&\le&m_{13}\\ 
\wedge && \wedge \\
m_{21}&\le&m_{22}\end{array}}
\frac{1}{m_{11}^{k_{11}}m_{12}^{k_{12}}m_{13}^{k_{13}}
m_{21}^{k_{21}}m_{22}^{k_{22}}}, \]
where we omit the condition $0<m_{ij}<N$ on running variables $m_{ij}$. 
Note that this generalizes both multiple harmonic and star harmonic sums in the sense that 
\begin{equation}\label{eq:special}
\zeta_{<N}\left(\ytableausetup{boxsize=1.5em}
\begin{ytableau}
k_1 \\
\vdots\\
k_r
\end{ytableau}\right)=\zeta_{<N}(k_1,\ldots,k_r), \qquad 
\zeta_{<N}\left(\ytableausetup{boxsize=1.5em}
\begin{ytableau}
k_1 & \cdots & k_r
\end{ytableau}\right)=\zeta_{<N}^\star(k_1,\ldots,k_r).
\end{equation}

In what follows, we assume that the index $\bk$ is \emph{diagonally constant}, 
i.e., $k_{ij}$ depends only on $i-j$. 
Hirose--Murahara--Onozuka \cite{HMO} assumed this condition 
to express the Schur multiple zeta value as the integral associated with a $2$-poset. 
Since our purpose in this section is to provide a finite sum analogue of their integral expression, 
it is natural to make the same assumption. 

We begin with preparing the following notation. 
For a diagonally constant index $\bk$ on $D$, we set 
\begin{align*}
&D_p\coloneqq \{(i,j)\in D\mid i-j=p\}, \\
&p_0\coloneqq \min\{p\in\Z\mid D_p\ne\emptyset\},\qquad p_1\coloneqq \max\{p\in\Z\mid D_p\ne\emptyset\}, \\
&k_p\coloneqq k_{ij} \text{ for any $(i,j)\in D_p$}. 
\end{align*}
As a matter of convention, we set $k_p=1$ for $p\in\Z$ with $D_p=\emptyset$. 
For each $p\in\Z$, there is a bijection 
\[D_p\longrightarrow J_p\coloneqq\{j\in\Z\mid (j+p,j)\in D_p\};\ (i,j)\longmapsto j. \]
Note that $J_p$ is an interval in $\Z$. Moreover, each pair $(J_p,J_{p+1})$ of intervals must satisfy the following condition. 

\begin{defn}
We say a pair $(J,J')$ of intervals in $\Z$ is \emph{consecutive} if $J=[j_0,j_1]$ and 
$J'=J$, $J\setminus \{j_1\}$, $J\cup\{j_0-1\}$ or $J\cup\{j_0-1\}\setminus \{j_1\}$. 
When $J$ is empty, this means that $J'$ is either empty or a singleton $\{j\}$ 
(in the latter case, we can take $(j_0,j_1)=(j+1,j)$). 
\end{defn}

For a tuple $\bm=(m_j)\in[1,N-1]^J$ of integers in $[1,N-1]$ indexed by an interval $J$, 
we set $\Pi(\bm)\coloneqq\prod_{j\in J}m_j$. 
Then the definition \eqref{eq:Schur MHS} of Schur multiple harmonic sum 
(for a diagonally constant index) is written as 
\[\zeta_{<N}(\bk)=\sum_{\substack{\bm_p\in[1,N-1]^{J_p}\\
(p_0\le p\le p_1)\\
\bm_{p_0}\vartriangleleft \cdots\vartriangleleft \bm_{p_1}}}
\prod_{p=p_0}^{p_1}\frac{1}{\Pi(\bm_{p})^{k_p}},\]
where $\bm_p$ runs over $[1,N-1]^{J_p}$ for each $p=p_0,\ldots, p_1$, satisfying the relation 
$\bm_{p_0}\vartriangleleft \cdots\vartriangleleft \bm_{p_1}$ defined as follows. 

\begin{defn}
Let $(J,J')$ be a consecutive pair of intervals. 
Let $\bm=(m_j)_{j\in J}$ and $\bn=(n_{j})_{j\in J'}$ be tuples of integers 
indexed by $J$ and $J'$ respectively. Then we define the relation 
\[\bm\vartriangleleft \bn\coloniff 
\begin{cases} 
m_j<n_j &\text{ if $j\in J$ and $j\in J'$,}\\
n_{j-1}\le m_j &\text{ if $j\in J$ and $j-1\in J'$}. 
\end{cases}\]
For later use, we also define 
\[\bm\trianglelefteq \bn\coloniff 
\begin{cases}
m_j\le n_j &\text{ if $j\in J$ and $j\in J'$,}\\
n_{j-1}<m_j &\text{ if $j\in J$ and $j-1\in J'$}. 
\end{cases}\]
\end{defn}

\begin{rem}\label{rem:n-1}
Notice that the relation $\trianglelefteq$ does \emph{not} mean ``$\vartriangleleft$ or equal''. 
Nevertheless, we have 
\begin{equation}\label{eq:n-1}
\bm \vartriangleleft \bn \iff \bm \trianglelefteq \bn-\bone
\end{equation}
where $\bn-\bone=(n_j-1)_{j\in J'}$. 
\end{rem}

\begin{ex}
If $J=J'=\{j\}$ (a singleton), then 
\[\bm\vartriangleleft \bn \iff m_j<n_j,\qquad 
\bm\trianglelefteq \bn \iff m_j\le n_j\]
holds. On the other hand, $J=\{j\}$ and $J'=\{j-1\}$, we have 
\[\bm\vartriangleleft \bn \iff n_{j-1}\le m_j,\qquad 
\bm\trianglelefteq \bn \iff n_{j-1}<m_j. \]
\end{ex}

Now we define the $\flat$-sum by using the above notation. 

\begin{defn}
For a diagonally constant index $\bk$ as above and an integer $N>0$, we set 
\[\zeta_{<N}^{\flat}(\bk)\coloneqq 
\sum_{\substack{\bn^{(l)}_p\in [1,N-1]^{J_p}\\
\bn^{(l)}_p\trianglelefteq \bn^{(l+1)}_p\;(1\le l<k_p)\\
\bn^{(k_p)}_p\vartriangleleft \bn^{(1)}_{p+1}\;(p_0\le p<p_1)}}
\prod_{p=p_0}^{p_1}\frac{1}{\Pi(\bN-\bn^{(1)}_p)\Pi(\bn^{(2)}_p)\cdots\Pi(\bn^{(k_p)}_p)}, \]
where $\bn^{(l)}_p$ runs over $[1,N-1]^{J_p}$ for each $p=p_0,\ldots,p_1$ and $l=1,\ldots,k_p$, 
satisfying the indicated relatioins. 
Moreover, for any tuple $\bn=(n_j)_{j\in J}\in[1,N-1]^J$, we set $\bN-\bn\coloneqq(N-n_j)_{j\in J}$. 
\end{defn}

Again, this generalizes both \eqref{eq:flat sum} and \eqref{eq:star flat sum}, that is, we have 
\begin{equation}\label{eq:special flat}
\zeta_{<N}^\flat\left(\ytableausetup{boxsize=1.5em}
\begin{ytableau}
k_1 \\
\vdots\\
k_r
\end{ytableau}\right)=\zeta_{<N}^\flat(k_1,\ldots,k_r), \qquad 
\zeta_{<N}^\flat\left(\ytableausetup{boxsize=1.5em}
\begin{ytableau}
k_1 & \cdots & k_r
\end{ytableau}\right)=\zeta_{<N}^{\star\flat}(k_1,\ldots,k_r). 
\end{equation}

Now we state the main theorem of this section. 

\begin{thm}\label{thm:MSW Schur}
For any diagonally constant index $\bk$ and any integer $N>0$, we have 
\begin{equation}\label{eq:MSW Schur}
\zeta_{<N}(\bk)=\zeta_{<N}^\flat(\bk). 
\end{equation}
\end{thm}

\begin{rem}
In view of identities \eqref{eq:special} and \eqref{eq:special flat}, 
\Cref{thm:MSW Schur} includes both Theorems \ref{thm:MSW} and \ref{thm:MSW star}. 
More generally, if the diagram $D$ is of anti-hook type, 
i.e., $D$ has the shape 
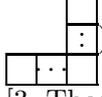
\[\ytableausetup{boxsize=5mm}\begin{ytableau}
\none & \none & \\
\none & \none & {\vdots} \\
 & {\cdots} & 
\end{ytableau}\]  
then the formula \eqref{eq:MSW Schur} is the finite sum analogue of the integral-series identity 
\cite[Theorem 4.1]{KY}. 
\end{rem}

We prove \Cref{thm:MSW Schur} by the method of \emph{connected sums}. 
Though this proof may look quite complicated, it is indeed a natural extension of the proof 
of  \Cref{thm:MSW} given in \cite{MSW}. 
The new ingredient is a trick using certain determinants. 
It is worthy of attention that a similar use of determinants appears in \cite{HMO}. 

For $0\le m,n\le N-1$, set 
\[C_N(m,n)\coloneqq \binom{n}{m}\biggm/\binom{N-1}{m}. \]
By definition, we have 
\begin{alignat*}{2}
C_N(0,n)&=1=C_N(m,N-1) &\quad & \text{for $0\le m,n\le N-1$}, \\
C_N(m,n)&=0 & & \text{for $0\le n<m\le N-1$}. 
\end{alignat*}

\begin{lem}\label{lem:C_N}
\begin{enumerate}
\item For $0<m<N$ and $0\le n\le n'<N$, 
\[\frac{1}{m}\bigl(C_N(m,n')-C_N(m,n)\bigr)=\sum_{b=n+1}^{n'} C_N(m,b)\frac{1}{b}. \]
\item For $0\le m\le m'<N$ and $0<n<N$,   
\[\sum_{a=m+1}^{m'}C_N(a,n)\frac{1}{n}=\bigl(C_N(m,n-1)-C_N(m',n-1)\bigr)\frac{1}{N-n}. \]
\end{enumerate}
\end{lem}
\begin{proof}
These are immediate consequences of the identities 
\[\frac{1}{m}\bigl(C_N(m,b)-C_N(m,b-1)\bigr)=C_N(m,b)\frac{1}{b}\]
and 
\[C_N(a,n)\frac{1}{n}=\bigl(C_N(a-1,n-1)-C_N(a,n-1)\bigr)\frac{1}{N-n},\]
respectively. 
\end{proof}

\begin{rem}
The relation (1) of the above lemma is the same as (4.1) of \cite[Lemma 4.1]{MSW}, 
though their symbol $C_N(n,m)$ corresponds to our $C_{N+1}(n,m)$. 
We also note that (4.2) of \cite[Lemma 4.1]{MSW} is obtained by 
combining (1) and (2) of our lemma. 
\end{rem}

Next we extend \Cref{lem:C_N} to certain determinants. 

\begin{defn}
Let $(J,J')$ be a pair of consecutive intervals, and 
$\bm=(m_j)_{j\in J}$ and $\bn=(n_{j'})_{j'\in J'}$ be tuples of elements in $[1,N-1]$ indexed by them. 
We set $\tJ=J\cup J'$ and define the symbols $\tm_j$ and $\tn_j$ for $j\in\tJ$ by 
\[\tm_j=\begin{cases}m_j & (j\in J),\\ 0 & (j\notin J),\end{cases} \qquad 
\tn_j=\begin{cases}n_j & (j\in J'),\\ N-1 & (j\notin J').\end{cases}\]
Then we define 
\[D_N(\bm,\bn)\coloneqq \det\bigl(C_N(\tm_{j_1},\tn_{j_2})\bigr)_{j_1,j_2\in\tJ}. \]
\end{defn}

\begin{lem}\label{lem:D_N empty}
When $J$ or $J'$ is empty, one has $D_N(\bm,\bn)=1$. 
\end{lem}
\begin{proof}
This is trivial when $J=J'=\emptyset$. 
If $J$ is empty and $J'$ is not, then $J'$ has to be a singleton. 
Hence $\bn$ is represented by a single element $n\in [1,N-1]$ and it holds that 
\[D_N(\bm,\bn)=C_N(0,n)=1. \]
Similarly, if $J'$ is empty and $J$ is a singleton, then $\bm$ is represented by $m\in[1,N-1]$ and it holds that 
\[D_N(\bm,\bn)=C_N(m,N-1)=1. \qedhere\]
\end{proof}

We say a tuple $\bm=(m_j)_{j\in J}$ of integers indexed by an interval $J$ is \emph{non-decreasing} 
if $m_j\le m_{j+1}$ holds whenever $j,j+1\in J$. 

\begin{lem}\label{lem:D_N}
Let $(J,J')$ be a consecutive pair of intervals. 
\begin{enumerate}
\item If $\bm\in[1,N-1]^{J}$ and $\bn\in[0,N-1]^{J'}$ are non-decreasing, we have 
\begin{equation}\label{eq:D_N 1}
\frac{1}{\Pi(\bm)}D_N(\bm,\bn)
=\sum_{\substack{\bb\in[1,N-1]^J\\ \bb\trianglelefteq\bn}} D_N(\bm,\bb)\frac{1}{\Pi(\bb)}. 
\end{equation}
\item If $\bm\in[0,N-1]^J$ and $\bn\in[1,N-1]^{J'}$ are non-decreasing, we have 
\begin{equation}\label{eq:D_N 2}
\sum_{\substack{\ba\in[1,N-1]^{J'}\\ \bm\vartriangleleft\ba}}D_N(\ba,\bn)\frac{1}{\Pi(\bn)}
=D_N(\bm,\bn-\bone)\frac{1}{\Pi(\bN-\bn)}. 
\end{equation}
\end{enumerate}
\end{lem}
\begin{proof}
First let us show \eqref{eq:D_N 1}. By shifting the numbering, we assume $J=[1,\oj]$. 
Put $C_{j_1j_2}=C_N(\tm_{j_1},\tn_{j_2})$ to simplify the notation. 
The proof is divided into two cases: (a) $0\notin J'$, (b) $0\in J'$. 

\begin{enumerate}
\item[(a)]
In this case, $\tJ=J=[1,\oj]$. If we set $\tn_0\coloneqq 0$, then we have 
\begin{align}
\notag
\frac{1}{\Pi(\bm)}D_N(\bm,\bn)
&=\det\biggl(\frac{1}{m_{j_1}}C_{j_1j_2}\biggr)_{j_1,j_2\in J}\\
\label{eq:D_N 1 proof}
&=\det\biggl(\frac{1}{m_{j_1}}\bigl(C_{j_1j_2}-C_{j_1(j_2-1)}\bigr)\biggr)_{j_1,j_2\in J}
\intertext{since $C_N(m,0)=0$ for any $m>0$. By \Cref{lem:C_N} (1), this is }
\notag
&=\det\Biggl(\sum_{\tn_{j_2-1}<b_{j_2}\le\tn_{j_2}}C_N(m_{j_1},b_{j_2})\frac{1}{b_{j_2}}\Biggr)_{j_1,j_2\in J}\\
\notag
&=\sum_{\substack{\bb\in[1,N-1]^J\\ \bb\trianglelefteq\bn}}D_N(\bm,\bb)\frac{1}{\Pi(\bb)} 
\end{align}
(note that the last equality holds regardless whether $J'=J$ or $J'\subsetneq J$, 
since in the latter case, the extra condition $b_{\oj}\le\tn_{\oj}=N-1$ holds automatically).
Thus we have shown \eqref{eq:D_N 1} in the case (a). 

\item[(b)]
In this case, $\tJ=[0,\oj]$. Since $C_N(0,n)=1$ for any $n$, we have 
\begin{align*}
\frac{1}{\Pi(\bm)}D_N(\bm,\bn)
&=\frac{1}{\Pi(\bm)}\begin{vmatrix}
1 & 1 & \cdots &1\\
C_{10} & C_{11} & \cdots & C_{1\oj}\\
\vdots & \vdots &  & \vdots\\
C_{\oj0} & C_{\oj1} & \cdots & C_{\oj\,\oj}
\end{vmatrix}\\
&=\frac{1}{\Pi(\bm)}\begin{vmatrix}
1 & 0 & \cdots &0\\
C_{10} & C_{11}-C_{10} & \cdots & C_{1\oj}-C_{1(\oj-1)}\\
\vdots & \vdots &  & \vdots\\
C_{\oj0} & C_{\oj1}-C_{\oj0} & \cdots & C_{\oj\,\oj}-C_{\oj(\oj-1)}
\end{vmatrix}\\
&=\det\biggl(\frac{1}{m_{j_1}}\bigl(C_{j_1j_2}-C_{j_1(j_2-1)}\bigr)\biggr)_{j_1,j_2\in J} 
\end{align*}
which is the expression in \eqref{eq:D_N 1 proof} and the rest of the proof is the same. 
\end{enumerate}

Next we show \eqref{eq:D_N 2}. By renumbering if necessary, we assume $J'=[1,\oj]$. 
Put $C'_{j_1j_2}=C_N(\tm_{j_1},\tn_{j_2}-1)$. 
Again we consider two cases: (a) $\oj+1\notin J$ or (b) $\oj+1\in J$. 

\begin{enumerate}
\item[(a)] 
In this case, $\tJ=J'=[1,\oj]$. If we set $\tm_{\oj+1}\coloneqq N-1$, we have 
\begin{align}
\notag 
D_N(\bm,\bn-\bone)\frac{1}{\Pi(\bN-\bn)}
&=\det\biggl(C'_{j_1j_2}\frac{1}{N-n_{j_2}}\biggr)_{j_1,j_2\in J'}\\
\label{eq:D_N 2 proof}
&=\det\biggl(\bigl(C'_{j_1j_2}-C'_{(j_1+1)j_2}\bigr)\frac{1}{N-n_{j_2}}\biggr)_{j_1,j_2\in J'}
\intertext{since $C_N(N-1,n-1)=0$ for any $n\in[1,N-1]$. By \Cref{lem:C_N} (2), this is }
\notag
&=\det\Biggl(\sum_{\tm_{j_1}<a_{j_1}\le\tm_{j_1+1}}C_N(a_{j_1},n_{j_2})\frac{1}{n_{j_2}}\Biggr)_{j_1,j_2\in J'}\\
\notag
&=\sum_{\substack{\ba\in[1,N-1]^{J'}\\ \bm\vartriangleleft\ba}}
D_N(\ba,\bn)\frac{1}{\Pi(\bn)} 
\end{align}
(again, the last equality holds regardless whether $J=J'$ or $J\subsetneq J'$). 
Thus we have shown \eqref{eq:D_N 2} in the case (a). 

\item[(b)] 
In this case, $\tJ=[1,\oj+1]$. Since $C_N(m,N-1)=1$ for any $m\in[0,N-1]$, we have 
\begin{align*}
D_N(\bm,\bn-\bone)\frac{1}{\Pi(\bN-\bn)}
&=\begin{vmatrix}
C'_{11} & \cdots & C'_{1\oj} & 1 \\
\vdots &  & \vdots & \vdots \\
C'_{\oj1} & \cdots & C'_{\oj\,\oj} & 1 \\
C'_{(\oj+1)1} & \cdots & C'_{(\oj+1)\oj} &1
\end{vmatrix}\frac{1}{\Pi(\bN-\bn)}\\
&=\begin{vmatrix}
C'_{11}-C'_{21} & \cdots & C'_{1\oj}-C'_{2\oj} & 0 \\
\vdots &  & \vdots & \vdots \\
C'_{\oj1}-C'_{(\oj+1)1} & \cdots & C'_{\oj\,\oj}-C'_{(\oj+1)\oj} & 0 \\
C'_{(\oj+1)1} & \cdots & C'_{(\oj+1)\oj} &1
\end{vmatrix}\frac{1}{\Pi(\bN-\bn)}\\
&=\det\biggl(\bigl(C'_{j_1j_2}-C'_{(j_1+1)j_2}\bigr)\frac{1}{N-n_{j_2}}\biggr)_{j_1,j_2\in J'}, 
\end{align*}
which is the expression in \eqref{eq:D_N 2 proof}, and the rest of the proof is the same. 
\end{enumerate}
The proof is complete. 
\end{proof}

Now let us introduce the following connected sum 
\begin{equation*}
\begin{split}
Z(\bk;q)\coloneqq \sum_{\substack{\bm_p\; (p_0\le p\le q)\\
\bn^{(l)}_p\; (q<p\le p_1)}}
&\Biggl(\prod_{p=p_0}^q \frac{1}{\Pi(\bm_{p})^{k_p}}\Biggr)\cdot D_N(\bm_q,\bn^{(1)}_{q+1}-\bone)\\
\cdot&\Biggl(\prod_{p=q+1}^{p_1} 
\frac{1}{\Pi(\bN-\bn^{(1)}_{p})\Pi(\bn^{(2)}_{p})\cdots \Pi(\bn^{(k_p)}_{p})}\Biggr)
\end{split}
\end{equation*}
for $q=p_0-1,p_0,\ldots,p_1$. 
Here $\bm_p\in [1,N-1]^{J_p}$ for $p_0\le p\le q$ run satisfying 
\[\bm_{p_0}\vartriangleleft\cdots\vartriangleleft \bm_q, \]
and $\bn^{(l)}_p\in [1,N-1]^{J_p}$ for $q<p\le p_1$ and $1\le l\le k_p$ run satisfying 
\[\bn^{(l)}_p\trianglelefteq \bn^{(l+1)}_p\ \text{ for $1\le l<k_p$ and }\ 
\bn^{(k_p)}_p\vartriangleleft \bn^{(1)}_{p+1}\ \text{for $q<p<p_1$}. \]
We also set $\bm_{p_0-1}$ and $\bn_{p_1+1}^{(1)}$ to be the empty tuple, 
which is compatible with the fact that $J_{p_0-1}$ and $J_{p_1+1}$ are the empty set. 

\begin{proof}[Proof of \Cref{thm:MSW Schur}]
By using \Cref{lem:D_N empty}, one can check that  
\[\zeta_{<N}(\bk)=Z(\bk;p_1),\qquad \zeta_{<N}^\flat(\bk)=Z(\bk;p_0-1). \]
On the other hand, for $p_0\le q\le p_1$, one shows the equality $Z(\bk;q)=Z(\bk;q-1)$ 
by applying \eqref{eq:D_N 1} $k_q$ times and then \eqref{eq:D_N 2} once 
(the equivalence \eqref{eq:n-1} is also used at the first application of \eqref{eq:D_N 1}). 
Thus we have 
\[\zeta_{<N}(\bk)=Z(\bk;p_1)=Z(\bk;p_1-1)=\cdots=Z(\bk;p_0-1)=\zeta_{<N}^\flat(\bk), \]
as desired. 
\end{proof}

\section{Hoffman's duality identity}
For a non-empty index (in the usual sense) $\bk=(k_1,\ldots,k_r)$, 
we define the \emph{multiple sum of Hoffman's type} by 
\[H_{<N}(\bk)=\sum_{1\le m_1\le\cdots\le m_r<N}\frac{(-1)^{m_r-1}}{m_1^{k_1}\cdots m_r^{k_r}}
\binom{N-1}{m_r}. \]
On the other hand, the \emph{Hoffman dual} of the index $\bk$ is defined by 
\[\bk^\vee=(\underbrace{1,\ldots,1}_{k_1}+\underbrace{1,\ldots,1}_{k_2}+\cdots
+\underbrace{1,\ldots,1}_{k_r}).\]
In other words, $\bk^\vee$ is obtained from 
\[\bk=(\underbrace{1+\cdots+1}_{k_1},\underbrace{1+\cdots+1}_{k_2},\ldots,\underbrace{1+\cdots+1}_{k_r})\]
by changing plus symbols to commas and vice versa. 

\begin{thm}[Hoffman's duality identity]\label{thm:Hoffman duality}
For any non-empty index $\bk=(k_1,\ldots,k_r)$ and any integer $N>0$, we have 
\begin{equation}\label{eq:Hoffman duality}
\zeta_{<N}^\star(\bk^\vee)=H_{<N}(\bk). 
\end{equation}
\end{thm}

This theorem was proved by Hoffman \cite{H} and independently by Kawashima \cite{K1}. 
Some different proofs are also known, e.g., one based on the integral expression \cite{Y} and 
one by the connected method \cite{SY}.

One can combine this identity \eqref{eq:Hoffman duality} 
with the MSW formula \eqref{eq:MSW star} for $\zeta_{<N}^\star(\bk^\vee)$, 
and make the ``change of variables'' $n_{ij}\leftrightarrow N-n_{ij}$. 
Then one obtains the following MSW-like expression of the multiple sum of Hoffman's type: 

\begin{prop}\label{prop:MSW H}
We have 
\begin{equation}\label{eq:MSW H}
H_{<N}(\bk)=H_{<N}^{\flat}(\bk), 
\end{equation}
where 
\[H_{<N}^{\flat}(\bk)\coloneqq 
\sum_{\substack{0<n_{i1}\le\cdots\le n_{ik_i}<N\,(1\le i\le r)\\
n_{(i-1)1}\le n_{ik_i}\,(2\le i\le r)}}
\frac{1}{n_{r1}\cdots n_{rk_r}}\prod_{i=1}^{r-1}\frac{1}{(N-n_{i1})n_{i2}\cdots n_{ik_i}}. \]
\end{prop}

The relation of the identities \eqref{eq:MSW star}, \eqref{eq:Hoffman duality} and \eqref{eq:MSW H} is 
summarized as follows: 
\[\xymatrix{
\zeta_{<N}^\star(\bk^\vee) 
\ar@{=}[r]^{\eqref{eq:MSW star}}  
\ar@{=}[d]_{\eqref{eq:Hoffman duality}} & 
\zeta_{<N}^{\star\flat}(\bk^\vee) 
\ar@{=}[d]^{n_{ij}\leftrightarrow N-n_{ij}}\\
H_{<N}(\bk) 
\ar@{=}[r]^{\eqref{eq:MSW H}} & 
H_{<N}^{\flat}(\bk) 
}\]

It is worth noticing that one can also prove \eqref{eq:MSW H} directly, 
because one then obtains yet another proof of Hoffman's identity \eqref{eq:Hoffman duality}. 
For this, one starts from showing that 
\begin{align*}
\sum_{m_1=1}^{m_2}\frac{1}{m_1^{k_1}}
&=\sum_{1\le m_1\le m_2}\frac{1}{m_1^{k_1}}\bigl(C_N(m_1,N-1)-C_N(m_1,0)\bigr)\\
&\overset{(1)}{=}\sum_{\substack{1\le m_1\le m_2\\ 1\le n_{1k_1}<N}}\frac{1}{m_1^{k_1-1}}
\bigl(C_N(m_1,n_{1k_1})-C_N(m_1,0)\bigr)\frac{1}{n_{1k_1}}\\ 
&\overset{(1)}{=}\sum_{\substack{1\le m_1\le m_2\\ 1\le n_{1(k_1-1)}\le n_{1k_1}<N}}\frac{1}{m_1^{k_1-2}}
\bigl(C_N(m_1,n_{1(k_1-1)})-C_N(m_1,0)\bigr)\frac{1}{n_{1(k_1-1)}n_{1k_1}}\\ 
&\overset{(1)}{=}\cdots
\overset{(1)}{=}\sum_{\substack{1\le m_1\le m_2\\ 1\le n_{11}\le\cdots\le n_{1k_1}<N}}
\bigl(C_N(m_1,n_{11})-C_N(m_1,0)\bigr)\frac{1}{n_{11}\cdots n_{1k_1}}\\
&\overset{(2)}{=}\sum_{1\le n_{11}\le\cdots\le n_{1k_1}<N}
\bigl(C_N(0,n_{11}-1)-C_N(m_2,n_{11}-1)\bigr)\frac{1}{(N-n_{11})n_{12}\cdots n_{1k_1}}\\
&=\sum_{1\le n_{11}\le\cdots\le n_{1k_1}<N}
\bigl(C_N(m_2,N-1)-C_N(m_2,n_{11}-1)\bigr)\frac{1}{(N-n_{11})n_{12}\cdots n_{1k_1}}. 
\end{align*}
Here one uses \Cref{lem:C_N} (1) and (2) as indicated. 
Note also that $C_N(m_1,N-1)=1$, $C_N(m_1,0)=0$ and $C_N(0,n_{11}-1)=1=C_N(m_2,N-1)$. 
By repeating such transformations, one arrives at 
\begin{align*}
H_{<N}(\bk)=\sum_{\substack{0<n_{i1}\le\cdots\le n_{ik_i}<N\,(1\le i\le r)\\
n_{(i-1)1}\le n_{ik_i}\,(2\le i\le r)}}
&\sum_{m_r=1}^{N-1}(-1)^{m_r-1}\binom{N-1}{m_r}C_N(m_r,n_{r1})\\
&\cdot\frac{1}{n_{r1}\cdots n_{rk_r}}\prod_{i=1}^{r-1}\frac{1}{(N-n_{i1})n_{i2}\cdots n_{ik_i}}. 
\end{align*}
Then, since 
\[\sum_{m_r=1}^{N-1}(-1)^{m_r-1}\binom{N-1}{m_r}C_N(m_r,n_{r1})
=\sum_{m_r=1}^{n_{r1}}(-1)^{m_r-1}\binom{n_{r1}}{m_r}=1 \]
by the binomial theorem, the equality \eqref{eq:MSW H} follows. 

\begin{rem}
Beginning with $\zeta^\star_{<N}(\bk)$ instead of $H_{<N}(\bk)$ and following the same computation, 
one arrives at $\zeta^{\star\flat}_{<N}(\bk)$ 
(in this case, the final step using the binomial theorem is not necessary). 
This direct proof of \Cref{thm:MSW star} is just a specialization of the proof of \Cref{thm:MSW Schur} 
given in the previous section. 
\end{rem}

\section{Relationship with Kawashima's identity}

In this section, we relate \Cref{thm:MSW star} to an identity due to Kawashima. 
For a non-empty index $\bk=(k_1,\ldots,k_r)$ and a complex variable $z$ with $\Re(z)>-1$, 
consider the nested series 
\begin{align*}
G_\bk(z)\coloneqq 
\sum_{\substack{0<m_{11}\le\cdots\le m_{1k_1}\\
\phantom{0}<m_{21}\le\cdots\le m_{2k_2}\\
\vdots\\
\phantom{0}<m_{r1}\le\cdots\le m_{rk_r}}}
&\prod_{i=1}^{r-1}\frac{1}{(m_{i1}+z)\cdots(m_{i(k_i-1)}+z)m_{ik_i}}\\
\cdot&\frac{1}{(m_{r1}+z)\cdots(m_{r(k_r-1)}+z)}\biggl(\frac{1}{m_{rk_r}}-\frac{1}{m_{rk_r}+z}\biggr). 
\end{align*}
Here $m_{ij}$ runs over all positive integers satisfying the indicated inequalities, 
hence $G_\bk(z)$ is an \emph{infinite} series. For example, 
\[G_{2,1,2}(z)
=\sum_{0<m_1\le m_2<m_3<m_4\le m_5}
\frac{1}{(m_1+z)m_2\,m_3(m_4+z)}\biggl(\frac{1}{m_5}-\frac{1}{m_5+z}\biggr). \]
This definition differs from Kawashima's original one in that we take the Hoffman dual of the index. 
In our notation, Kawashima's identity is stated as follows. 

\begin{prop}[{\cite[Proposition 3.2]{K2}}]\label{prop:Kawashima}
For any non-empty index $\bk$ and any integer $N>0$, 
\begin{equation}\label{eq:Kawashima}
G_{\overleftarrow{\bk}}(N-1)=\zeta_{<N}^\star(\bk)
\end{equation}
holds, where  $\overleftarrow{\bk}\coloneqq(k_r,\ldots,k_1)$ for $\bk=(k_1,\ldots,k_r)$.  
\end{prop}

\begin{rem}
The main result of the article \cite{K2} is the identity $G_{\overleftarrow{\bk}}(z)=F_\bk(z)$, 
where $F_\bk(z)$ is a function (called the \emph{Kawashima function}) characterized as 
the Newton series that interpolates multiple star harmonic sums: $F_\bk(N-1)=\zeta_{<N}^\star(\bk)$. 
Thus the above \Cref{prop:Kawashima} states that $G_{\overleftarrow{\bk}}(N-1)=F_\bk(N-1)$ holds for any integer $N>0$. 
This is an important step in Kawashima's proof of the identity $G_{\overleftarrow{\bk}}(z)=F_\bk(z)$ as functions. 
\end{rem}

Now let us combine \Cref{prop:Kawashima} with \Cref{thm:MSW star}. 
The result is the following: 

\begin{prop}\label{prop:G flat}
For any non-empty index $\bk$ and any integer $N>0$, 
\begin{equation}\label{eq:G flat}
G_{\overleftarrow{\bk}}(N-1)=\zeta_{<N}^{\star\flat}(\bk)
\end{equation}
holds. 
\end{prop}

There is a triangle 
\[\xymatrix{
\zeta_{<N}^{\star}(\bk)
\ar@{=}[rr]^{\eqref{eq:MSW star}}
\ar@{=}[rd]_{\eqref{eq:Kawashima}}
&& \zeta_{<N}^{\star\flat}(\bk) 
\ar@{=}[ld]^{\eqref{eq:G flat}} \\
& G_{\overleftarrow{\bk}}(N-1) & 
}\]
of three quantities, and we have deduced the equality \eqref{eq:G flat} from the other two. 
On the other hand, \eqref{eq:G flat} can be shown directly, by using the transformations 
\begin{equation}\label{eq:trans1}
\begin{split}
&\frac{1}{m'+N-1}\sum_{m=m'}^\infty\biggl(\frac{1}{m+N-1-n''}-\frac{1}{m+N-n'}\biggr)\\
&=\frac{1}{m'+N-1}\sum_{n=n'}^{n''}\frac{1}{m'+N-1-n}
=\sum_{n=n'}^{n''}\biggl(\frac{1}{m'+N-1-n}-\frac{1}{m'+N-1}\biggr)\frac{1}{n} 
\end{split}
\end{equation}
and 
\begin{equation}\label{eq:trans2}
\begin{split}
&\frac{1}{m'}\sum_{m=m'+1}^\infty\biggl(\frac{1}{m+N-1-n''}-\frac{1}{m+N-n'}\biggr)\\
&=\frac{1}{m'}\sum_{n=n'}^{n''}\frac{1}{m'+N-n}
=\sum_{n=n'}^{n''}\biggl(\frac{1}{m'}-\frac{1}{m'+N-n}\biggr)\frac{1}{N-n} 
\end{split}
\end{equation}
repeatedly. For instance, starting from 
\[G_{2,2}(N-1)=\sum_{0<m_1\le m_2<m_3\le m_4}
\frac{1}{(m_1+N-1)m_2(m_3+N-1)}\biggl(\frac{1}{m_4}-\frac{1}{m_4+N-1}\biggr), \]
one proceeds 
\begin{align*}
\frac{1}{m_3+N-1}\sum_{m_4=m_3}^\infty&\biggl(\frac{1}{m_4}-\frac{1}{m_4+N-1}\biggr)\\
\overset{\eqref{eq:trans1}}{=}
\sum_{n_4=1}^{N-1}&\biggl(\frac{1}{m_3+N-1-n_4}-\frac{1}{m_3+N-1}\biggr)\frac{1}{n_4}, 
\displaybreak[1]\\
\frac{1}{m_2}\sum_{m_3=m_2+1}^\infty&\biggl(\frac{1}{m_3+N-1-n_4}-\frac{1}{m_3+N-1}\biggr)\\
\overset{\eqref{eq:trans2}}{=}
\sum_{n_3=1}^{n_4}&\biggl(\frac{1}{m_2}-\frac{1}{m_2+N-n_3}\biggr)\frac{1}{N-n_3}, 
\displaybreak[1]\\
\frac{1}{m_1+N-1}\sum_{m_2=m_1}^\infty&\biggl(\frac{1}{m_2}-\frac{1}{m_2+N-n_3}\biggr)\\
\overset{\eqref{eq:trans1}}{=}
\sum_{n_2=n_3}^{N-1}&\biggl(\frac{1}{m_1+N-1-n_2}-\frac{1}{m_1+N-1}\biggr)\frac{1}{n_2}, 
\end{align*}
and ends with a telescopic sum 
\[\sum_{m_1=1}^\infty\biggl(\frac{1}{m_1+N-1-n_2}-\frac{1}{m_1+N-1}\biggr)
=\sum_{n_1=1}^{n_2}\frac{1}{N-n_1}. \]
The result is 
\[G_{2,2}(N-1)=
\sum_{n_4=1}^{N-1}\sum_{n_3=1}^{n_4}\sum_{n_2=n_3}^{N-1}\sum_{n_1=1}^{n_2}
\frac{1}{(N-n_1)n_2(N-n_3)n_4}
=\zeta_{<N}^{\star\flat}(2,2). \]

Thus we obtain a new proof of \Cref{prop:Kawashima} 
through \Cref{thm:MSW star} and \Cref{prop:G flat}. 
Conversely, Propositions \ref{prop:Kawashima} and \ref{prop:G flat} implies \Cref{thm:MSW star}, 
and hence \Cref{thm:MSW}. 
This could be another possible way of finding the MSW formula! 

\section*{Acknowledgment}
The author would like to thank Professor Shin-ichiro Seki for carefully reading the first draft 
and giving many valuable comments, and thank Yuto Tsuruta for pointing out that 
it is necessary to reverse the index in Proposition 5.1 and subsequent arguments. 
The author also would like to thank the referee for giving helpful suggestions. 
This research was supported by JSPS KAKENHI JP18H05233 and JP21K03185. 



\begin{thebibliography}{99}
\bibitem{HMS}
M.~Hirose, T.~Matsusaka and S.~Seki, 
A discretization of the iterated integral expression of the multiple polylogarithm, 
preprint, 2024, arXiv:2404.15210. 

\bibitem{HMO}
M.~Hirose, H.~Murahara and T.~Onozuka, 
Integral expressions for Schur multiple zeta values, 
\textit{Indag.\ Math.} (2024). 
\url{https://doi.org/10.1016/j.indag.2024.05.010}

\bibitem{H}
M.~E.~Hoffman, 
Quasi-symmetric funtions and mod $p$ multiple harmonic sums, 
\textit{Kyushu J.\ Math.} \textbf{69} (2015), 345--366. 

\bibitem{KY}
M.~Kaneko and S.~Yamamoto, 
A new integral-series identity of multiple zeta values and regularizations, 
\textit{Selecta Math.\ New Series} \textbf{24} (2018), 2499--2521. 

\bibitem{K1}
G.~Kawashima, 
A class of relations among multiple zeta values, 
\textit{J.\ Number Theory} \textbf{129} (2009), 755--788. 

\bibitem{K2}
G.~Kawashima, 
Multiple series expressions for the Newton series which interpolate 
finite multiple harmonic sums, 
preprint, 2009, arXiv:0905.0243. 

\bibitem{MSW}
T.~Maesaka, S.~Seki and T.~Watanabe, 
Deriving two dualities simultaneously from a family of identities 
for multiple harmonic sums, 
preprint, 2024, arXiv:2402.05730. 

\bibitem{NPY}
M.~Nakasuji, O.~Phuksuwan and Y.~Yamasaki, 
On Schur multiple zeta functions: A combinatoric generalization of multiple zeta functions, 
\textit{Adv.\ Math.} \textbf{333} (2018), 570--619.

\bibitem{S}
S.~Seki, 
A proof of the extended double shuffle relation without using integrals, 
preprint, 2024, arXiv:2402.18300. 

\bibitem{SY}
S.~Seki and S.~Yamamoto, 
Ohno-type identities for multiple harmonic sums, 
\textit{J.\ Math.\ Soc.\ Japan} \textbf{72} (2020), 673--686. 

\bibitem{Y}
S.~Yamamoto, 
Multiple zeta-star values and multiple integrals, 
\textit{RIMS K\^oky\^uroku Bessatsu} \textbf{B68} (2017), 3--14.  

\end{thebibliography}
\end{document}